\newif\ifpictures
\picturestrue

\documentclass[12pt]{amsart}
\usepackage{amssymb}
\usepackage{amsmath}
\usepackage{amscd}
\usepackage[pdftex]{graphicx}
\usepackage{hyperref}
\usepackage{bbm} % for the black board 1
\usepackage{bm} % \bm X gives bold X
\usepackage{enumitem}
\usepackage[textsize=tiny]{todonotes}
\usepackage{makecell}

\usepackage{xparse}

\numberwithin{equation}{section}
\newtheorem{thm}{Theorem}

\newtheorem{lemma}[thm]{Lemma}
\newtheorem{cor}[thm]{Corollary}

\theoremstyle{definition}

\newtheorem{example}[thm]{Example}
\newtheorem{remark1}[thm]{Remark}
\newtheorem{openproblem1}[thm]{Open problem}
\newtheorem{definition}[thm]{Definition}

\newenvironment{rem}{\begin{remark1}\rm}{\end{remark1}}

\numberwithin{thm}{section}

\newcounter{FNC}[page]
\def\newfootnote#1{{\addtocounter{FNC}{2}$^\fnsymbol{FNC}$%
     \let\thefootnote\relax\footnotetext{$^\fnsymbol{FNC}$#1}}}

\newcommand{\V}{\mathbb{V}}

\newcommand{\N}{\mathbb{N}}

\newcommand{\R}{\mathbb{R}}

\newcommand{\Z}{\mathbb{Z}}

\newcommand{\ones}{\mathbf{1}}
\newcommand{\primal}{\mathcal{P}}
\newcommand{\dual}{\mathcal{D}}

\DeclareMathOperator{\diag}{diag}

\DeclareMathOperator{\tr}{tr}

\DeclareMathOperator{\aux}{\mathrm{aux}}

\newcommand{\sym}{\mathcal{S}}

\DeclareDocumentCommand\OB{ g }{%
    { \mathcal{O}_B \IfNoValueF {#1} { \left( #1 \right) } } }

\DeclareDocumentCommand\OO{ g }{%
    {\mathcal{O} \IfNoValueF {#1} { \mleft( #1 \mright) } } }

\DeclareDocumentCommand\sOB{ g }{ %
    {\widetilde{\mathcal{O}}_B \IfNoValueF {#1} { \mleft( #1 \mright) } } }

\DeclareDocumentCommand\sOO{ g } {%
    {\widetilde{\mathcal{O}} \IfNoValueF {#1} { \mleft( #1 \mright) } } }

\title[Equivalence of SDP and zero-sum 
semidefinite games]{On the equivalence of semidefinite programming and zero-sum semidefinite
games}

\author[J.~Elliott]{Jesse Elliott}

\address{Jesse Elliott: Sorbonne Universit\'e and Inria Paris.
        IMJ-PRG, 4 place Jussieu,
        75005 Paris, France}

\email{jesse.elliott@inria.fr}

\author[C.~Ickstadt]{Constantin Ickstadt}

\author[T.~Theobald]{Thorsten Theobald}

\address{Constantin Ickstadt, Thorsten Theobald:
        Goethe-Universit\"at, FB 12 -- Institut f\"ur Mathematik,
        Postfach 11 19 32, 60054 Frankfurt am Main, Germany}

\email{ickstadt@math.uni-frankfurt.de,
        theobald@math.uni-frankfurt.de}

\author[E.~Tsigaridas]{Elias Tsigaridas}

\address{Elias Tsigaridas: Sorbonne Universit\'e, Paris University, CNRS and Inria Paris.
        IMJ-PRG,  4 place Jussieu,
        75252 Paris Cedex 05, France}

\email{elias.tsigaridas@inria.fr}

\date{\today}

\begin{document}

\begin{abstract}
By results of Dantzig (1951) and Adler (2013), computing the optimal solutions
of a linear program is equivalent to finding optimal strategies in zero-sum
bimatrix games. Dantzig's original result was incomplete, in the sense that the
reduction of a linear program to a zero-sum game did not work for all possible
linear programs.

We show that, under a natural constraint qualification requiring either the
existence of strongly optimal primal-dual solutions 
or of a strictly unbounded direction, 
computing the solution of a semidefinite program is equivalent to
finding optimal strategies in an associated zero-sum semidefinite game. 
Our work builds upon Ickstadt, Theobald, and Tsigaridas (2024), where, similar to
Dantzig's work, the proposed reduction cannot handle a certain subclass of
semidefinite programs.
Our main proof ingredients for the equivalence result include: 
(i)~a semidefinite generalization of von Stengel's (2023) extension of Dantzig's
construction; 
(ii)~techniques for handling more general duality phenomena in the
semidefinite setting; and 
(iii)~an explicit bound for the (coordinates) of the solutions of a semidefinite program. 
As a by-product, the game value provides a
certificate: it is zero if and only if strongly optimal solutions exist, and otherwise
optimal strategies yield an infeasibility certificate for the primal or dual program.
\end{abstract}

\newpage

\maketitle

\section{Introduction}
Semidefinite programming provides an important generalization
of linear programming that allows one to solve optimization problems 
on the cone of positive semidefinite matrices~\cite{boyd2004convex,wolkowicz2000handbook}. 
It contains linear programming as a special case: a semidefinite
program (SDP) reduces to a linear program (LP)
when all constraint matrices are diagonal. This additional generality gives semidefinite programming strictly more expressive power than 
linear programming; it can model a significantly broader 
class of problems arising, for example, in combinatorial optimization~\cite{goemans1995maxcut, 
lovasz1979shannon}, polynomial 
optimization~\cite{blekherman2012sdp,ragopt-book}
or quantum information theory \cite{watrous-book}. However, the additional expressive power introduces complications that do not arise in linear programming, such as the possibility of a positive duality gap or non-attainment of the optimal value~\cite{pataki2000geometry,pataki2017bad,ramana1997exact}.

Linear programming has strong links to fundamental
branches in game theory and this close
connection has been studied continuously throughout the development 
of both fields. A key result states that 
computing an optimal solution to an LP is equivalent to finding optimal strategies in zero-sum bimatrix games. This was first observed by Dantzig \cite{dantzig-1951-equivalence} in 1951.
However, Dantzig pointed
out that his reduction of an LP to a zero-sum game did not work in
all cases. In 2012, Adler \cite{adler-2013} filled the gap by transforming a given LP into a linear feasibility problem and reducing that linear feasibility problem to a zero-sum game.
In this way, we can reduce any LP to a zero-sum game 
and we can recover a solution of the original LP from 
the optimal strategies of the game if the LP had optimal solutions to begin with. 
If the LP does not have an optimal solution, then the primal or the dual
program must be infeasible; in this case the reduction provides
certificates of the infeasibility.

In the more recent years,  
Brooks and Reny \cite{brooks_reny} and von Stengel \cite{vonstengel}
introduced alternative reductions of LPs from zero-sum games 
which are based on modified Dantzig constructions.
In these reductions, we can always recover solutions 
of the original LP, if they exist. 
Additionally, the authors of these works prove, using their constructions, 
that the theorem of strong duality in linear optimization follows from von Neumann's minimax theorem~\cite{von-neumann-1945}.

Also quite recently, a first step was taken to elevate the equivalence
of linear programming and zero-sum bimatrix games to the semidefinite
setting. Ickstadt et al.~\cite{semi-games}
introduced \textit{semidefinite games} in which each player's strategy is a density matrix, i.e., 
a positive semidefinite matrix with trace one. 
We can consider these games as the real-valued 
version of one-round quantum games, as described by Gutoski 
and Watrous~\cite{watrous1}.
A nonlinear extension of the classical Lemke-Howson
algorithm for computing Nash equilibria to the semidefinite case
appears in~\cite{itt-lemke-howson}.
Similarly to the classical matrix case, where
we can compute the optimal strategies and the value of a zero-sum game
using linear programming, for zero-sum semidefinite games one can 
employ semidefinite programming.

Based on Dantzig's construction, a reduction of SDPs
to zero-sum semidefinite games appears in \cite{semi-games}, where an optimal strategy of the game is (up to scaling) a solution of the initial SDP.
However, this reductions has the  same limitation as Dantzig's original work;
that is, there exist cases in which a solution of the SDP cannot be recovered.
For the set of LPs or SDPs for which 
this happens, we do not know an explicit description.
For a further extension regarding conic programming we refer to the work of Dimou~\cite{dimou-equiv-2024}.

Our goal is to show that we can elevate the improvements 
of Dantzig's construction for the LP case to the semidefinite setting. We introduce a modified reduction of an SDP 
to a zero-sum semidefinite game that is based on von~Stengel's approach 
for LPs \cite{vonstengel}. For any given SDP we construct a 
\textit{modified semidefinite Dantzig game} $G_M$, where $M$ 
is a \emph{solution bound}, that is, a constant large enough to dominate the optimal
solution of an auxiliary SDP  associated with the input; 
see~\eqref{eq:def:M}. 
The role of $M$ is similar to the bound employed in von Stengel's modification of Dantzig's
construction in the linear case: it forces the relevant coordinate
of an optimal strategy to be positive and thereby makes it possible
to recover optimal primal/dual solutions from the game. 
When the input data consists of rational numbers, 
then we derive an explicit value for $M$ from coordinate bounds 
for the KKT system of the SDP; see Section~\ref{se:M}. 
The game $G_M$ has value zero if and only if a pair of strongly optimal 
solutions of the original SDP exists
and, in this case, we can recover the solutions from optimal strategies of $G_M$. 
We present the details in Theorem~\ref{th:reduction}. 

As a consequence,
if the constraint qualification
of assuming strongly optimal solutions is satisfied
(see Definition~\ref{de:constraint-qualification}),
then the result provides a full equivalence of a pair of bounded SDPs to a
semidefinite game, see Corollary~\ref{co:summary}.
The result also provides a characterization of
whether strongly optimal solutions for a given SDP
exist. We discuss the constraint qualification in detail in Section~\ref{se:reduction}.

In the case where the SDP or its dual has 
a strictly unbounded direction (this is formally
defined in Definition~\ref{def:strict-unbounded-direction}), the optimal strategies
of the modified semidefinite Dantzig game $G_M$ provide a certificate for the
unboundedness of that SDP, see Theorem~\ref{th:reduction-unbounded} and Corollary~\ref{co:summary}.
Similar to Theorem~\ref{th:reduction}, the constraint qualification
in Theorem~\ref{th:reduction-unbounded} is discussed in detail
at the beginning of Section~\ref{se:reduction}.

Compared to the earlier reduction of SDPs by Ickstadt et al.~\cite{semi-games},
the current work fills the gap in the reduction in a
way which generalizes the LP setting.
That is, whenever a pair of optimal solutions exists in an SDP with no duality gap, then we recover such solutions. 
This was not the case in the reduction by Ickstadt et al.~\cite{semi-games} 
as, similar to the limitations of Dantzig's original (LP) construction,
even for certain simple examples, the reduction does not 
recover the optimal solution.
We refer to Example~\ref{ex:limitation2} for a specific example.
Our explicit assumptions also implicitly underlie the work
of Ickstadt et al.~\cite{semi-games} and the cases covered there are
a subclass of the cases covered here. 
Only the cases that do not meet our constraint qualification
remain and, as we argue in detail at the beginning
of Section~\ref{se:reduction},
there is a good reason to remove these cases from the discussion.
Namely, the excluded SDPs have a degenerate duality
behavior that is also excluded in many works on SDPs.
Moreover, this degeneration is not caused by the viewpoint of games.
Similarly, we completely cover the class of 
unbounded SDPs which satisfy our constraint qualification.
This class was not completely captured by~\cite{semi-games}.

From the viewpoint of the techniques we employ, handling strictly unbounded SDPs requires substantially more work than handling bounded SDPs. In
particular, our proofs address the additional complications
that emanate from the possibility that the 
optimal value(s) in an auxiliary SDP 
(formally defined in Definition~\ref{def:aux-problem})
might not be attained.
Such situations cannot
arise in the reduction of LPs to zero-sum matrix games
and thus were not addressed in \cite{vonstengel}.

An important ingredient in the construction is the (solution) bound $M$
that is large enough to dominate the optimal solution of an associated 
auxiliary SDP; see~\eqref{eq:def:M}.
For this, we exploit the bounds on the isolated roots of sparse and structured 
polynomial systems \cite{emt-dmm-jsc-20}.
If the input consists of matrices with rational entries,
then we show that we can bound $M$ explicitly
in terms of the bitsize of the input, say $\tau$,
and the problem dimensions, $m$ and $n$.
The bitsize bound is polynomial in $\tau$ and exponential in $m$ and $n$.
The bound is nearly optimal (Example~\ref{ex:opt-M-bd}).
\medskip

The paper is structured as follows. Section~\ref{se:prelim} provides background
on linear and semidefinite programming, the equivalence of linear programming
with zero-sum bimatrix games and on semidefinite games. 
In Section~\ref{se:reduction} we introduce the modified Dantzig game for
semidefinite programming and address the equivalence of semidefinite programming
to semidefinite games.
Section~\ref{se:M} provides explicit coordinate bounds for the 
solutions of the associated KKT system and uses them to derive
an explicit value of the solution bound $M$
for rational input data.
Section~\ref{se:examples} gives some examples and 
we close the paper with a short outlook in
Section~\ref{se:outlook}.

\section{Preliminaries}\label{se:prelim}

We denote the set of symmetric real $n\times n$ matrices by $\sym_n$. 
A matrix $X \in \sym_n$ is called \emph{positive semidefinite}
if all its eigenvalues are nonnegative. In this case, we write $X\succeq 0$ and we call the set of positive semidefinite matrices $\sym_n^+$. A matrix $X$ is \emph{negative semidefinite} if $-X$ is positive semidefinite. We write the all-ones vector as $\ones$ whenever the dimension is clear from the context. 
Set $[n]:=\{1, \ldots, n\}$. In addition, $\R_+$ denotes the set of nonnegative real numbers. 

\subsection{Linear and semidefinite programming}

A \textit{linear program} (LP) is an optimization problem of finding a vector which minimizes a given linear objective function under affine-linear constraints. We consider linear programs in the primal normal form
\begin{equation}
 \label{eq:lp:primal}
 \min_{x\in \R^n} \{ c^Tx \, : \, Ax\ge b, \, x\ge 0 \} \, ,
\end{equation}
where $c\in \R^n, b \in \R^m, A\in \R^{m\times n}$.
Its dual program is
\begin{equation}
  \label{eq:lp:dual}
  \max_{y\in \R^m} \{ b^Ty \, : \, A^T y \le c, \, y\ge 0 \} \, .
\end{equation}

A \emph{semidefinite program} (SDP) is an optimization problem of finding
a vector which minimizes a linear objective function over an affine-linear 
slice of the positive semidefinite cone. We consider semidefinite
programs in the primal normal form
\begin{eqnarray} \label{eq:SDP_primal-prelim} 
&& \inf_{X} \{ \langle C , X \rangle : \, \langle A_i, X\rangle \ge b_i, \; i \in [m], \, X \in \sym_n^+  \},
\end{eqnarray}
where $A_1, \ldots, A_m, C \in \sym_n$ and $b\in \R^m$.
Its dual SDP is
\begin{eqnarray}
  \sup_{y} \{ b^Ty : \, \sum_{i=1}^m y_i A_i \preceq  C, \, y\in \R^m_+   \} \label{eq:SDP_dual-prelim} \, .
\end{eqnarray}
Any vector which satisfies the constraints of a given LP or a given
SDP is called a \textit{feasible solution} of the problem. The problem 
is called \emph{feasible} if a feasible solution exists.

Unlike linear programming, semidefinite programming can have a pathological (or degenerate) behavior 
even when both the primal and dual programs are feasible, 
see e.g.~\cite{boyd2004convex,wolkowicz2000handbook}
for a detailed treatment.
In particular, strong duality can fail and the optimal value 
may not be attained by any feasible solution. 
Such pathological SDPs have been studied by Pataki~\cite{pataki2000geometry} and Tun\c{c}el~\cite{tunccel2016polyhedral}, who 
characterize the facial structure of the positive semidefinite cone
that is responsible for the duality gaps and for the fact that an
optimum value might not be attained.
A sufficient condition that ensures that these degenerate cases do not occur
is \emph{Slater's condition}: the primal satisfies Slater's condition 
if there exists a strictly feasible point $X \succ 0$ with $\langle A_i, X \rangle 
\geq b_i$ for all $i \in [m]$, and the dual satisfies Slater's condition if there 
exists $y \geq 0$ with $\sum_{i=1}^m y_i A_i \prec C$. When both the primal and 
dual satisfy Slater's condition, strong duality holds and both the infimum and 
supremum are attained. 
Moreover, conditioned on feasibility, SDPs satisfying 
Slater's condition are \emph{generic}~\cite[Theorem~4.3]{djs-generic-conic-2017}.
We should interpret this as follows: in the space of the input data,
the set of matrices  $(C, A_1, \ldots, A_m, b)$
that do not satisfy Slater's condition are of measure zero.

\subsection{The equivalence of linear programming and zero-sum bimatrix games}
\label{ssse:Dantzig_reduction}

A \textit{bimatrix game} is a game between two players played on two matrices $A,B \in \R^{d_1\times d_2}$ where one player chooses a row $i\in [d_1]$ and the other player chooses a column $j\in [d_2]$ of $A$ and $B$ and the corresponding entries $A_{ij}$ and $B_{ij}$ define the payoffs to either player. The set of \emph{(mixed) strategies} 
is the set of probability vectors over the pure strategies, 
$\Delta_1 = \{ x\in \R^{d_1}_+ \, : \, \ones^T x = 1 \}$
and 
$\Delta_2 = \{ y\in \R^{d_2}_+ \, : \, \ones^T y = 1 \}$.
The \emph{(expected) payoffs} for the players are $p_1(x,y)=x^TAy$ and $p_2(x,y)=x^TBy$. A special case of bimatrix games are \textit{zero-sum} bimatrix games where 
one player's gain is the other player's loss, i.e., $B=-A$.

Optimal strategies of bimatrix games can be viewed as optimal solutions of a minimax optimization problem. Furthermore, if both players play optimal strategies, then the expected payoff of the
first player is independent of the exact choice of the optimal strategies. 
This payoff $v$ is called the \textit{value} of the game. It
satisfies the following minimax equation
which goes back to von Neumann \cite{von-neumann-1945},
\begin{equation}\label{eq:lp:minimax}
\min_{y\in \Delta_2}\max_{x\in \Delta_1} x^TAy =v = \max_{x\in \Delta_1}\min_{y\in \Delta_2}x^TAy \, .
\end{equation}
Von Neumann also proved that the problem of finding optimal solutions for any zero-sum 
bimatrix game can be reduced to solving an LP.
The reverse task of reducing the optimal solutions of an arbitrary LP to optimal strategies of an associated 
zero-sum bimatrix game was addressed by Dantzig \cite{dantzig-1951-equivalence}
in the following result.

\begin{thm}[Dantzig \cite{dantzig-1951-equivalence}]
\label{th:dantzig-lp}
Let $A\in \R^{m\times n},\, b\in \R^m$ and $c\in \R^n$ and consider the corresponding linear program with its primal program \eqref{eq:lp:primal} and dual program \eqref{eq:lp:dual}.
Construct the zero-sum bimatrix game defined by the payoff matrix 
$P \in \R^{(m+n+1)\times (m+n+1)}$ for the first player with
\begin{equation*}
P=\begin{pmatrix}
0 & A & -b\\
-A^T & 0 & c\\
b^T & -c^T &0
\end{pmatrix} \, .
\end{equation*}
Then there is a strategy $z=(y,x,t)\in \R^m \times \R^n \times \R$ which is an optimal strategy for both players and $t\cdot(b^Ty - c^Tx)=0$.

If $t>0$ then $\frac{1}{t}x$ is an optimal solution of the primal program \eqref{eq:lp:primal} and $\frac{1}{t}y$ is an optimal solution of the dual program \eqref{eq:lp:dual}.

If $b^Ty - c^Tx < 0$ then $x$ is
an unbounded direction of~\eqref{eq:lp:primal}
and~\eqref{eq:lp:dual} is infeasible or
$y$ is an unbounded direction of~\eqref{eq:lp:dual}
and~\eqref{eq:lp:primal} is infeasible.
\end{thm}

The zero-sum game in the theorem is called \emph{Dantzig game}. 
For later considerations, it is  useful to recall the proof.

\begin{proof}
The zero-sum game defined by $P$ is symmetric in the sense $P=-P^T$. Hence, by the 
minimax condition~\eqref{eq:lp:minimax}, its value $v$
is zero. Let $z=(y,x,t)\in \R^m \times \R^n \times \R$ be an optimal strategy for the second player. By the symmetry of the game, $z$ is also an optimal strategy for the first player. 
Since $v=0$, we see that $z^TP \ge 0$ and therefore
\begin{eqnarray*}
Ax -tb & \ge & 0,\\
-A^Ty + tc & \ge & 0,\\
b^Ty -c^Tx & \ge & 0.
\end{eqnarray*}

Assume $t>0$. Then $b^Ty -c^Tx = 0$, since otherwise one of the players could achieve a positive payoff and $z$ would not be an optimal strategy for both players. Now consider the pair $(\frac{1}{t}x,\frac{1}{t}y)$. The first two inequalities ensure that 
$\frac{1}{t}x$ and $\frac{1}{t}y$ are feasible solutions for the primal and dual program, respectively, and the third inequality proves that they are indeed optimal solutions, since both the primal and the dual program attain the same value. This shows the first claim.

Now assume that $b^Ty - c^Tx < 0$. Then $t=0$ 
and the desired assertion holds.
\end{proof}

Theorem~\ref{th:dantzig-lp} 
provides optimal solutions for the original LP or a certificate that such solutions do not exist, because at least one of the programs is infeasible. However, there is one case in which the
reduction does not apply and this occurs if $t=0=b^Ty - c^Tx$. Indeed, there exist LPs for which 
we cannot make a statement about the solutions using this reduction. This limitation 
of the reduction was already pointed out by Dantzig.

\begin{example}
\label{ex:limitation1}
A simple example where Dantzig's construction fails
is the pair of LPs with
$A = \left(
  \begin{smallmatrix} 1 & 1 \\ 1 & 1 
  \end{smallmatrix}
  \right)
$,
$b=(1,0)^T$, $c=(1,0)^T$, which has primal and dual optimal
value zero. The strategy $(y,x,t)=(0,1,0,0,0)$ is optimal for both players, but due
to $t=0$, Theorem~\ref{th:dantzig-lp} does not make a statement.
\end{example}

The incompleteness in Theorem~\ref{th:dantzig-lp}
was first eradicated by Adler \cite{adler-2013}, by initially transforming the LP to 
a linear feasibility problem. Variants of such a reduction were given by Brooks and 
Reny \cite{brooks_reny}
as well as by von Stengel \cite{vonstengel}. 
We recall the latter one, since it will be relevant for us.
The construction \textit{always} provides optimal solutions for the original LP or a certificate that such solutions do not exist because at least one of the programs is infeasible. 
This closes the gap in Dantzig's original reduction.

\begin{definition}
Let $A\in \R^{m\times n},\, b\in \R^m$ and $c\in \R^n$ and consider the corresponding linear program with 
its primal program \eqref{eq:lp:primal} and dual program \eqref{eq:lp:dual}. 
Construct the auxiliary LP
\begin{equation}\label{local:aux-problem-LP}
\min_{x \in \R_+^n,y\in \R_+^m,w\in \R_+} \{w \, : \, Ax -w\ones \le b, \, A^T y -w\ones \le c, \, b^Ty - c^Tx -w \le 0 \}
\end{equation}
with some optimal solution $(x^*,y^*,w^*)$.

Let $M\in \R$ such that $\ones^Tx^* + \ones^Ty^* +1 \le M $. 
Then, the 
\emph{modified Dantzig game} 
is the zero-sum bimatrix game defined by the payoff matrix
\begin{equation}
\label{eq:pm}
P_M \ = \ \begin{pmatrix}
0 & A & -b\\
-A^T & 0 & c\\
b^T & -c^T &0\\
\ones^T & \ones^T & -M
\end{pmatrix} \ \in \ \R^{(m+n+2) \times (m+n+1)}.
\end{equation}
\end{definition}

\begin{thm}[von Stengel \cite{vonstengel}]
Let $A\in \R^{m\times n},\, b\in \R^m$ and $c\in \R^n$ and consider the corresponding linear program with its primal program \eqref{eq:lp:primal} and dual program \eqref{eq:lp:dual}. Let $(x^*,y^*,w^*)$ be an optimal solution of the
auxiliary LP~\eqref{local:aux-problem-LP}.
Then, for every $M \in \R$ such that $\ones^Tx^* + \ones^Ty^* +1 \le M$,
the modified Dantzig game has a nonnegative value. There are two distinct cases.
\begin{enumerate}
\item
If $v=0$, let $(y,x,t)$ be an optimal strategy for the second player. 
Then $t>0$ and $\frac{1}{t}x$ is an optimal solution of the primal program \eqref{eq:lp:primal} and $\frac{1}{t}y$ is an optimal solution of the dual program \eqref{eq:lp:dual}.
Moreover, $(t^*x^*,t^*y^*)$ is an optimal strategy for the second player with $t^*=\frac{1}{\ones^Tx^* + \ones^T y^*+1}$.

\item
If $v>0$, let $(y,x,t,u)$ be an optimal strategy for the first player. Then $t=0$, $u=v$ and 
$x$ is and unbounded direction 
for~\eqref{eq:lp:primal} or
$y$ is an unbounded direction 
for~\eqref{eq:lp:dual}, which shows that the primal or dual program is infeasible.
Moreover, $v<1$ and $w^*=\frac{v(M+1)}{1-v}$.
\end{enumerate}
\end{thm}

A major improvement of this reduction compared to Dantzig's is that, through the introduction of the solution bound $M$, the last column is played with positive probability in every optimal strategy. Therefore, it is always possible to divide by this probability and recover optimal solutions of the original LP if those solutions existed in the first place.

\subsection{Semidefinite games}

We consider semidefinite games between two players with the strategy spaces of real-valued density matrices $\mathcal{X} = \{ X \in \sym_{d_1}^+ \, : \, \tr(X) = 1\}$ and $\mathcal{Y} = \{Y \in \sym^+_{d_2} \, : \, \tr(Y) = 1\}$, where $\tr$ denotes
the trace of a matrix. See \cite{semi-games} or 
\cite{itt-semidefinite-network}.
The payoff functions of a semidefinite game on $\sym_{d_1} \times \sym_{d_2}$
are
\begin{eqnarray*}
  p_A(X, Y) \ = \ \sum_{i,j,k,l} X_{ij} A_{ijkl} Y_{kl}
  \; \text{ and } \;
  p_B(X, Y) \ = \ \sum_{i,j,k,l} X_{ij} B_{ijkl} Y_{kl},
\end{eqnarray*}
where $A=(A_{ijkl})_{i,j \in [d_1], l,k \in [d_2]}$ and $B=(B_{ijkl})_{i,j \in [d_1], l,k \in [d_2]}$ 
are the payoff tensors satisfying the symmetry conditions
\[
  \begin{array}{rccccccc}
  & A_{ijkl} & = & A_{jikl} & = & A_{ijlk} & = & A_{jilk} \\
  \text{ and } & B_{ijkl} & = & B_{jikl} & = & B_{ijlk} & = & B_{jilk} \, .
  \end{array}
\]

A semidefinite game is called \emph{zero-sum} if $p_B(X,Y) = -p_A(X,Y)$ for all $X \in \mathcal{X}$
and $Y \in \mathcal{Y}$ and it can be described by the payoff function to the first
player.
A zero-sum semidefinite game on $\sym_d \times \sym_d$ for $d\in \N$ with payoff function $p$
for the first player is called \textit{symmetric} if $p(X,Y)=-p(Y,X)$ holds for all 
strategies $X,Y$. In that case, the value of the zero-sum semidefinite game is zero 
\cite[Lemma~4.3]{semi-games} and in particular, $p(X,X) =0$ for all $X \in \mathcal{X}$.

By \cite[Theorem 4.1]{semi-games}, the set of optimal strategies of
both players in a zero-sum semidefinite games can be expressed 
as the solutions of a semidefinite program.

\section{The reduction of semidefinite programs to semidefinite games}\label{se:reduction}

Throughout the paper, we use the following setup. Given
$A_1, \ldots, A_m, C \in \sym_n$ and $b\in \R^m$,
denote by $\primal$ the semidefinite program in the normal form
\begin{eqnarray} \label{eq:SDP_primal} 
&& \primal \; : \; \inf_{X} \{ \langle C , X \rangle : \, \langle A_i, X\rangle \ge b_i, \; i \in [m], \, X \in \sym_n^+  \}
\end{eqnarray}
and by $\dual$ its dual
\begin{eqnarray}
  \dual \; : \;
  \sup_{y} \{ b^Ty : \, \sum_{i=1}^m y_i A_i \preceq  C, \, y\in \R^m_+   \} \label{eq:SDP_dual} \, .
\end{eqnarray}

We explain our model and the assumptions which we make.
In line with the earlier work on optimization and game theory
mentioned in the introduction, our goal is to construct from 
one primal-dual pair $\mathcal{P}$ and $\mathcal{D}$ 
a zero-sum semidefinite game aiming at the following properties.
If both $\mathcal{P}$ and $\mathcal{D}$ are feasible and bounded, 
then the set 
of optimal strategies coincides with the set of optimal points of 
the pair. If $\mathcal{P}$ or $\mathcal{D}$ is
unbounded then the set of optimal strategies 
provides a certificate for the unboundedness. 

In view of the possible degeneracies of semidefinite programming
duality, we make the following assumptions. First, consider
the situation of a primal-dual pair $\mathcal{P}$ and $\mathcal{D}$
which are both feasible and bounded. Since in a zero-sum
semidefinite game, the optimal strategies are always attained,
we make the assumption that in $\mathcal{P}$ and $\mathcal{D}$,
the optimal points are attained. Moreover, since the game
is intended to capture both the primal and the dual, we
make the assumption that the primal and the dual optimal
value coincide. Indeed, in our main Theorem~\ref{th:reduction}
below
we will recognize from the value of the game whether both
conditions on the pair of SDPs are satisfied. If 
$\mathcal{P}$ and $\mathcal{D}$ satisfy Slater's condition
then our two assumptions are satisfied.

In the situation  of a primal-dual pair $\mathcal{P}$ and
$\mathcal{D}$ where one of them, say, $\mathcal{P}$ is unbounded,
we make a similar assumption. We then assume that there
exists a \emph{strictly} unbounded direction defined
as follows. 

\begin{definition} \label{def:strict-unbounded-direction}
Given a pair of SDPs $(\primal, \dual)$, 
we call a matrix $W\in \sym^{+}_n$ a \emph{strictly unbounded direction
for} $\primal$ (or \emph{strictly primal unbounded direction}) if
\begin{equation} 
\langle A_i,W\rangle > 0 \text{ for } i \in [m] \text{ and } \langle C,W\rangle <0 .
\end{equation}

We call $y\in \R^m_+$ a \emph{strictly unbounded direction for} 
$\dual$
(or \textit{strictly dual unbounded direction}) if
\begin{equation} 
\sum_{i=1}^m y_i A_i \prec 0 \text{ and } b^Ty>0 .
\end{equation}
\end{definition}

\begin{rem}
Similar to the notion of a strictly unbounded direction one
could also define the \emph{strict infeasibility} of an
SDP. While both $\mathcal{P}$ and $\mathcal{D}$ can be
infeasible, note that it is impossible that both are simultaneously
strictly infeasible.
\end{rem}

A strictly unbounded direction for $\primal$ or $\dual$ ensures that that 
program is unbounded and, in particular, feasible.
By the weak duality theorem,
a pair $(X,y) \in \sym_n^+ \times \R_+^m$ is strongly optimal for $\primal$ and $\dual$
if and only if
\[
  \langle A_i, X\rangle \ge b_i \text{ for } i \in [m], \; \, 
  \sum_{i=1}^m y_i A_i \preceq  C 
  \, \text{ and } \, 
  \langle C,X \rangle - b^T y \le 0.
\]
In order to come up with a feasible SDP even if $\primal$ or $\dual$ is not feasible, we consider an auxiliary SDP 
in which all the inequalities are relaxed. 

\begin{definition} \label{def:aux-problem}
Given a pair $(\primal, \dual)$, 
the \emph{primal auxiliary SDP} $\primal_{aux}$ is defined as
\begin{equation} \label{eq:w_defi_SDP_primal}
\begin{split}
\primal_{\aux} \, : \; \inf &\quad w \\
\text{s.t.} &\quad \langle A_i, X\rangle + w \ge b_i \quad \text{for } i \in [m],\\
&\quad \sum_{i=1}^m  y_i A_i - w I_n \preceq C,  \\
&\quad \langle C,X \rangle - b^Ty -w \le 0, \\
&\quad X \succeq 0, y \ge0, w\ge 0  .
\end{split}
\end{equation}
Its dual, called the \emph{dual auxiliary SDP} $D_{\aux}$, is
\begin{equation}
\label{eq:w_defi_SDP_dual}
\begin{split}
\dual_{\aux} \, : \; \sup &\quad b^Tz - \langle C,W\rangle\\
\text{s.t.} &\quad \sum_{i=1}^m  z_i A_i - rC \preceq 0, \\
&\quad \langle A_i, W\rangle - rb_i \ge 0 \quad \text{for } i \in [m],\\
&\quad \ones^T z + \tr(W) + r \le 1,\\
&\quad W \succeq 0, z \ge0, r\ge 0.
\end{split}
\end{equation}
\end{definition}

Note that the infimum in $\primal_{\aux}$ is attained with optimal value zero
if and only if there exists a pair $(X,y)$ of strongly optimal solutions 
to $\primal$ and $\dual$.

\begin{rem} \label{rem:strictly-dual}
The auxiliary program $\primal_{\aux}$ in Definition \ref{def:aux-problem} always has interior points. 
This can be seen by choosing $w$ sufficiently large. Hence, the supremum of the dual program $\dual_{\aux}$ is 
always attained. The infimum in the primal $\primal_{\aux}$ might not be attained.
\end{rem}

\begin{lemma}
\label{le:unbounded-direction}
If $\primal$ or $\dual$ has a strictly unbounded direction,
then the infimum in $\primal_{\aux}$ is attained.
\end{lemma}

\begin{proof}
Let $\bar{W}$ be a strictly primal unbounded direction of $\primal$. That is, 
$\bar{W} \in \sym_n^+$ with 
$\langle A_i, \bar{W} \rangle > 0$ for $i \in [m]$
and $\langle C, \bar{W} \rangle < 0$. By suitable scaling,
we can assume $\tr(\bar{W}) \le 1$.
Now consider the relaxation of the primal auxiliary
SDP $\primal_{\aux}$ defined by

\begin{equation}
\label{eq:relaxed-aux-problem-primal}
\inf \left\{ w \, : \, \sum\limits_{i=1}^m  y_i A_i - w I_n \preceq C, \,
y \ge 0, \, w \ge 0\right\} \, .
\end{equation}
Its dual is
\begin{equation}
  \label{eq:relaxed-aux-problem-dual}
  \sup \left\{ - \langle C, W \rangle \, : \,
  \langle A_i, W\rangle \ge 0 \text{ for } i \in [m], \,
  \tr(W) \le 1, \,
  W \succeq 0
\right\}.
\end{equation}

The primal relaxation~\eqref{eq:relaxed-aux-problem-primal}
still has interior points. Moreover, $\bar{W}$ is
an interior point
of the dual~\eqref{eq:relaxed-aux-problem-dual} 
with positive objective value.
Hence, there exists a pair of optimal solutions 
$(\bar{y},\bar{w})$ to~\eqref{eq:relaxed-aux-problem-primal}
and $\bar{W}$ to~\eqref{eq:relaxed-aux-problem-dual} 
such that $\bar{w}=-\langle C, \bar{W} \rangle>0$.

We carry over this pair of optimal solutions to a pair
of optimal solutions for $\primal_{\aux}$ and $\dual_{\aux}$. To this end,
set $\bar{z}:=0\in \R^m$, $\bar{r}:=0$ and $\bar{X}:=\lambda \bar{W}$ 
for some sufficiently large $\lambda >0$ such that 
$\langle C,\bar{X}\rangle - b^T \bar{y} - \bar{w} <0$.
Then the point $(\bar{X},\bar{y},\bar{w})$
is feasible for $\primal_{\aux}$ and the point
$(\bar{W},\bar{z},\bar{r})$ is feasible for $\dual_{\aux}$.

Further, the objective values of $(\bar{X},\bar{y},\bar{w})$
in~$\primal_{\aux}$ and of $(\bar{W},\bar{z},\bar{r})$ in $\dual_{\aux}$
coincide, so that both the supremum in $\dual_{\aux}$ is attained
(which was already known by Remark~\ref{rem:strictly-dual})
and the infimum in $\primal_{\aux}$ is attained.

If there exists a strictly dual unbounded direction $y$ 
of $\dual$, we can consider 
instead the relaxation of the auxiliary 
primal $\primal_{\aux}$, defined by
$
\label{internal:eq:relaxed-aux-problem}
\inf \{ w \, : \, 
\langle A_i, X\rangle + w \ge b_i \text{ for } i\in [m], \,
X \succeq 0, \, w\ge 0 \} \, .
$
Its dual is
$
\sup \{ b^T z \, : \,
 \sum_{i=1}^m z_i A_i \preceq 0, \,
 \ones^T z \le 1, \,
 z\ge 0\}.
%\end{split}
$
By a similar argument to the previous case, both this primal
and its dual have interior points and the resulting optimal pair
of solutions carry over to $\primal_{\aux}$ and $\dual_{\aux}$,
which shows the claim.
\end{proof}

Building on \cite{semi-games} and \cite{vonstengel}, 
we now introduce a reduction of the SDP~\eqref{eq:SDP_primal} to a modified semidefinite
Dantzig game $G_M$ (for an appropriate constant $M$), 
which we define in the following.
We make use of an upper bound $M$ for the size of an optimal solution of
the primal auxiliary problem $\primal_{\aux}$ if such a solution exists.
If the infimum of $\primal_{\aux}$ is attained, then let $M>0$ such that 
\begin{equation} \label{eq:def:M}
\tr(X^*) + \ones^T y^* +1 \le M
\end{equation}
holds for some optimal solution $(X^*,y^*,w^*)$. Note that this condition 
does not need to hold for every optimal solution or even for every
feasible solution.
If the infimum of $\primal_{\aux}$ 
is not attained, we can arbitrarily choose $M>0$. Any constant $M$
satisfying the desired property is called a \emph{solution bound}
for the pair of SDPs $(\primal,\dual)$. 
Section~\ref{se:M} shows how to obtain such a solution bound explicitly for
rational input data by bounding isolated primal/dual solutions of a KKT system.

Given a pair $(\primal,\dual)$ with a solution bound $M$,
we define the zero-sum semidefinite game $G_M$ via the payoff tensor $Q_M$ on $\sym_{n+m+2}\times \sym_{n+m+1}$.
$Q_M$ exhibits a block structure which can be described as follows. 
The strategies $\bar{X}$ and $\bar{Y}$ can be viewed as diagonal block matrices 
of the form
\begin{equation} \label{eq:strategies}
\bar{X} = \left(
    \begin{array}{c|c|c|c}
      X^{(1)} & 0 & 0 & 0\\
      \hline
      0 & \diag(y^{(1)}) & 0 &0\\
      \hline
      0 & 0 & t^{(1)}&0\\
      \hline
      0 & 0 & 0 & u
    \end{array}
    \right), \quad \bar{Y}=\left(
    \begin{array}{c|c|c}
      X^{(2)} & 0 & 0\\
      \hline
      0 & \diag(y^{(2)}) & 0 \\
      \hline
      0 & 0 & t^{(2)}
    \end{array}
    \right)
\end{equation}
with $y^{(1)},y^{(2)}\in \R^m_+, \, X^{(1)},X^{(2)}\in \sym_n^+, t^{(1)},t^{(2)}\ge 0, u\ge 0$ and $\tr(\bar{X})=\tr(\bar{Y})=1$. In order to achieve this block structure on 
$\bar{X}$, every entry of $Q_M$ whose first index corresponds to an entry outside one of the
non-zero blocks in $\bar{X}$ or whose second index corresponds to an entry outside one
of the non-zero blocks in $\bar{Y}$ is zero. This induces a block structure on $Q_M$ and thus on $\bar{X}$.
The payoff tensor $Q_M$ can be conveniently defined in terms of the bilinear payoff function
on the block strategy matrices. The payoffs of the first player in 
the zero-sum semidefinite game $G_M$ are given by the bilinear
function
\begin{equation} \label{eq:payoff_G_m}
\begin{array}{rcl}
p(\bar{X},\bar{Y}) & = & \sum\limits_{i=1}^m y_i^{(1)} (t^{(2)} b_i -\langle A_i , X^{(2)} \rangle ) + \langle X^{(1)}, \sum\limits_{i=1}^m y_i^{(2)}A_i -  t^{(2)}C \rangle \\
& & + \ t^{(1)}(\langle C, X^{(2)}\rangle - b^Ty^{(2)})
 + u \left( \tr(X^{(2)})+ \ones^T y^{(2)} - t^{(2)}M \right) \, .
\end{array}
\end{equation}

Note that except for the addition of the last row and column to the strategy
matrix of the first player, this is the semidefinite Dantzig game defined 
on $\sym_{n+m+1}\times \sym_{n+m+1}$ as introduced in \cite{semi-games}.
The construction~\eqref{eq:payoff_G_m} also generalizes the modified Dantzig 
game~\eqref{eq:pm} from the LP setting.

We denote the subgame of the semidefinite game 
$G_M$ on $\sym_{n+m+1}\times \sym_{n+m+1}$ by $G$.
$G$ is symmetric and hence its value is zero.
The value of $G_M$ is nonnegative because the first player can guarantee a nonnegative payoff by choosing 
$u=0$ in~\eqref{eq:strategies}
and by just playing an optimal strategy of the game $G$. We can now prove our first
main result.

\begin{thm}[Equivalence theorem, bounded case] \label{th:reduction}
Consider a pair $(\primal,\dual)$ and its modified Dantzig game $G_M$.

A pair $(X,y)$ of strongly optimal solutions of $(\primal,\dual)$ exists 
if and only if the value of the game $G_M$ is zero.

Furthermore, if the value of $G_M$ is zero, then any optimal strategy of the
second player for $G_M$ will be of the form $\diag(X,y,t)$ with $t>0$ and $\frac{1}{t}X$, $\frac{1}{t}y$ are strongly optimal solutions for $(\primal,\dual)$.
Here, $\diag(X,y,t)$ denotes the
block matrix with the diagonal blocks
$X$, $\diag y$ and $t$.
\end{thm}

\begin{example}
\label{ex:limitation2}
A simple example where the reduction 
in~\cite{semi-games} failed and which
is covered by our modified Dantzig construction is the semidefinite version
of Example~\ref{ex:limitation1}. That is,
all payoffs corresponding to non-diagonal
strategy elements are set to zero.
\end{example}

\begin{proof}[Proof of Theorem~\ref{th:reduction}]
First assume that a strongly optimal solution exists. By definition of
the solution bound $M$, there exists an optimal solution $(X',y',0)$ of the auxiliary 
SDP $\primal_{\aux}$ that satisfies \eqref{eq:def:M}.
Let $t':= \frac{1}{\tr(X')+ \ones^T y' +1}$ and consider the strategy 
\[
  Y' \ = \ \left(
    \begin{array}{c|c|c}
      t'X' & 0 & 0\\
      \hline
      0 & t'\diag(y') & 0 \\
      \hline
      0 & 0 & t'
    \end{array}
    \right)
\]
of the second player.
For any strategy $\bar{X}$ of the first player, the payoff is
\begin{eqnarray*}
p(\bar{X},Y') & = & \sum\nolimits_{i=1}^m \underbrace{y_i^{(1)}t'}_{\ge 0} \underbrace{(b_i -\langle A_i , X' \rangle )}_{\le 0} \ + \ \langle \underbrace{X^{(1)}}_{\succeq 0}, \underbrace{t'\sum\nolimits_{i=1}^m (y_i'A_i - C)}_{\preceq 0} \rangle \\
& & + \underbrace{t^{(1)} \,t'}_{\ge 0} \underbrace{(\langle C, X'\rangle - b^Ty')}_{\le 0}
 \ + \ \underbrace{u \, t'}_{\ge 0} \underbrace{(\tr(X')+ \ones^T y' - M)}_{\le 0} \ \le \ 0 \, ,
\end{eqnarray*}
where we used that $(X',y',0)$ is an optimal solution of $\primal_{\aux}$.
Hence, the strategy $Y'$ of the second player ensures a nonnegative payoff 
for her, which implies that the game value $v$ satisfies $v \le 0$.
Since $v \ge 0$ by the initial remarks before the theorem, we conclude $v=0$.

Conversely, assume that the game $G_M$ has value $v=0$.
Then an optimal strategy $\bar{Y}$ of the second player exists such that the payoff
$p(\bar{X},\bar{Y})$ of the first player is nonpositive for any strategy $\bar{X}$.
Hence $t^{(2)}> 0$, because otherwise player 1 could set $u=1$ in her 
strategy and receive the payoff $\tr(X^{(2)})+ \ones^Ty^{(2)} = 1 >0$ in
contraction to $v=0$.
Similar arguments on the possible responses of the first player show 
that the following inequalities also hold.
\begin{eqnarray}\label{eq:feasible_X}
(t^{(2)} b_i -\langle A_i , X^{(2)} \rangle ) &\le & 0 \quad \text{for } i\in [m], \\\label{eq:feasible_y}
 \sum_{i=1}^m y_i^{(2)}A_i -  t^{(2)}C &\preceq & 0, \\
\label{eq:optimality}
\langle C, X^{(2)}\rangle - b^Ty^{(2)} &\le & 0.
\end{eqnarray}
We claim that $\frac{1}{t^{(2)}}X^{(2)}$ and $\frac{1}{t^{(2)}}y^{(2)}$ 
form a pair of strongly optimal solutions of $(\primal,\dual)$. Indeed, the inequalities \eqref{eq:feasible_X} and \eqref{eq:feasible_y} ensure that $\frac{1}{t^{(2)}}X^{(2)}$ and $\frac{1}{t^{(2)}}y^{(2)}$ are feasible solutions and \eqref{eq:optimality} ensures that strong duality holds and that the strategies are optimal. 
\end{proof}

\begin{rem}
\label{re:paux-daux}
If the pair $(\primal,\dual)$ has an optimal solution, then the preceding
proof employs an optimal solution $(X',y',0)$ of $\primal_{\aux}$
satisfying~\eqref{eq:def:M}. Setting $W=X'$, $z=y'$ and $r=1$ shows that
a pair of strongly optimal solutions exist 
for $\primal_{\aux}$ and $\dual_{\aux}$.
\end{rem}

The following example illustrates the reduction.

\begin{example} \label{ex:bounded}
Consider the pair of SDPs $(\primal, \dual)$ with
\[
\begin{array}{rcl}
\primal & : & \inf  \left\{ \left\langle \begin{pmatrix}
1 & 2\\
2 & 2
\end{pmatrix}, X\right\rangle \, : \, \left\langle \begin{pmatrix}
1 & 1\\
1 & 0
\end{pmatrix}, X\right\rangle \ge 1, \, X\succeq 0 \right\} ,
\\
\dual & : & \sup  \left\{ y \, : \, \begin{pmatrix}
1-y & 2-y\\
2-y & 2
\end{pmatrix} \succeq 0, \, y\ge 0 \right\} .
\end{array}
\]
The infimum of the auxiliary SDP 
\begin{equation*} 
\begin{array}{rcl}
    \primal_{aux}\,: & \inf & \left \{ w \, : \, 
     \left\langle 
     \begin{pmatrix}
1 & 1\\
1 & 0
\end{pmatrix}, X\right\rangle  +  w \ge 1, \;
y \begin{pmatrix}
1 & 1\\
1 & 0
\end{pmatrix} - w I_2 \preceq \begin{pmatrix}
1 & 2\\
2 & 2
\end{pmatrix}, \right. \; \\
& & 
\left.
\left\langle \begin{pmatrix}
1 & 2\\
2 & 2
\end{pmatrix},X \right\rangle - y -w  \le  0, \;
X \succeq 0, y \ge0, w\ge 0 \right\} 
\end{array}
\end{equation*}
is attained for $(X^*,y^*,w^*)= \left( 
  \left( \begin{smallmatrix}
    1 & 0\\
    0 & 0
\end{smallmatrix} \right),1,0\right)$, so we can choose, say, $M=3$, as a solution bound.
The modified Dantzig game $G_M$ is now defined by the payoff
\begin{eqnarray*}
p(\bar{X},\bar{Y}) & = &  y^{(1)} \left(t^{(2)} -\left\langle \begin{pmatrix}
1 & 1\\
1 & 0
\end{pmatrix} , X^{(2)} \right\rangle \right) + \left\langle X^{(1)}, y^{(2)}\begin{pmatrix}
1 & 1\\
1 & 0
\end{pmatrix} -  t^{(2)}\begin{pmatrix}
1 & 2\\
2 & 2
\end{pmatrix} \right\rangle \\
& & + t^{(1)} \left( \left\langle \begin{pmatrix}
1 & 2\\
2 & 2
\end{pmatrix}, X^{(2)}\right\rangle - y^{(2)} \right)
 + u \left(\tr(X^{(2)})+ y^{(2)} - 3t^{(2)}\right) .
\end{eqnarray*}
An optimal strategy for the second player is $\diag(X^{(2)},y^{(2)},t^{(2)})$ where $ X^{(2)} =\frac{1}{3}
\left( \begin{smallmatrix}
    1 & 0\\
    0 & 0
\end{smallmatrix} \right)$, 
$y^{(2)}=\frac{1}{3}$ and $t^{(2)}=\frac{1}3$.
From this we can recover the pair of strongly optimal solutions 
$(X,y)= \left( \left( \begin{smallmatrix}
    1 & 0\\
    0 & 0
\end{smallmatrix} \right),1\right)$ for $(\primal,\dual)$.
\end{example}

\begin{thm}[Equivalence theorem, unbounded case] \label{th:reduction-unbounded}
Consider a pair $(\primal,\dual)$ and its modified Dantzig game $G_M$. Assume that a strictly unbounded direction exists in either $\primal$ or $\dual$.

Then, for every optimal strategy $\diag (X',y',t',u) $ for player 1 in $G_M$,
we have $t'=0$ and $u=v$.

Furthermore, $X'$ is an unbounded
direction for $\primal$ or
$y'$ is an unbounded direction 
for $\dual$, 
which is a 
certificate that $\dual$ or $\primal$ 
is infeasible.
\end{thm}

First, we illustrate the statement with an example.
\begin{example} 
Let $(\primal,\dual)$ be the pair of SDPs given by
\[
\begin{array}{rcl}
\primal & : & \inf \left\{ \left\langle \begin{pmatrix}
0 & -1\\
-1 & 0
\end{pmatrix}, X\right\rangle \, : \, \left\langle \begin{pmatrix}
-1 & 1\\
1 & 0
\end{pmatrix}, X\right\rangle \ge 1, \, X\succeq 0 \right\}, \\
\dual & : & \sup \left\{ y \, : \, \begin{pmatrix}
y & -1-y\\
-1-y & 0
\end{pmatrix}\succeq 0, \, y\ge 0 \right\} .
\end{array}
\]
$\primal$ is unbounded and $\dual$ is infeasible.
The infimum of the auxiliary program
\begin{equation*} 
\begin{array}{rcl}
    \primal_{aux}\,: & \inf & \left \{ w \, : \, 
     \left\langle 
     \begin{pmatrix}
-1 & 1\\
1 & 0
\end{pmatrix}, X\right\rangle  +  w \ge 1, \;
 \begin{pmatrix}
y & -1-y\\
-1-y & 0
\end{pmatrix} + w I_2 \succeq 0, \right. \; \\
& & 
\left.
\left\langle \begin{pmatrix}
0 & -1\\
-1 & 0
\end{pmatrix},X \right\rangle - y -w  \le  0, \;
X \succeq 0, y \ge0, w\ge 0 \right\} 
\end{array}
\end{equation*}
is one and an optimal solution is 
$(X^*,y^*,w^*)= ( \left( \begin{smallmatrix}
    0 & 0\\
    0 & 0
\end{smallmatrix} \right),0,1)$, so we can choose, say $M=1$, as a solution bound.
The modified Dantzig game $G_M$ is defined by the payoff
\begin{eqnarray*}
p(\bar{X},\bar{Y}) & = &  y^{(1)} \left(t^{(2)} -\left\langle \begin{pmatrix}
-1 & 1\\
1 & 0
\end{pmatrix} , X^{(2)} \right\rangle \right) + \left\langle X^{(1)}, y^{(2)}\begin{pmatrix}
-1 & 1\\
1 & 0
\end{pmatrix} -  t^{(2)}\begin{pmatrix}
0 & 1\\
1 & 0
\end{pmatrix} \right\rangle \\
& & + t^{(1)}\left( \left\langle \begin{pmatrix}
0 & 1\\
1 & 0
\end{pmatrix}, X^{(2)}\right\rangle - y^{(2)}\right)
 + u (\tr(X^{(2)})+ y^{(2)} - t^{(2)}).
\end{eqnarray*}
An optimal strategy for the second player in $G_M$ is 
$\diag(X^{(2)},y^{(2)},t^{(2)})$, where 
$X^{(2)} =\frac{1}{9}\left( \begin{smallmatrix}
    1 & \frac{3}{2}\\
    \frac{3}{2} & 5
\end{smallmatrix} \right),
y^{(2)}=0$ and $t^{(2)}=\frac{1}{3}$.
Further, $X^{(2)}$ is a strictly unbounded direction of $\primal$.

Note that the reduction from \cite{semi-games}, which does not use the solution bound, allows for the optimal strategy given by $X^{(2)}=\left( \begin{smallmatrix}
    0 & 0\\
    0 & 1
\end{smallmatrix} \right), \, y^{(2)}=0, \, t^{(2)}=0$  for player 2. This does not give us any information about the optimal solutions of the SDP since $t^{(2)}=0=\langle C,X^{(2)}\rangle - b^Ty^{(2)}$ with $C=\left( \begin{smallmatrix}
    0 & -1\\
    -1 & 0
\end{smallmatrix} \right)$ and $b=1$ which is a case that is explicitly not covered in \cite{semi-games}.
\end{example}

In order to prove Theorem \ref{th:reduction-unbounded},
we begin with some useful observations.

\begin{lemma} \label{le:payoff_bound}
Given a pair $(\primal,\dual)$, 
assume that its modified Dantzig game $G_M$ has value $v\ge 0$ and an optimal 
strategy $\bar{X} = \diag(X',y',t',u)$ for the first player. Then
\[
\label{eq:le:payoff_bound_v}
\begin{array}{rcl}
\langle X', A_i \rangle - t'b_i &\ge & v-u \quad \text{for } i \in [m], \\
 \displaystyle t'C - \sum_{i=1}^m y_i'A_i  &\succeq & (v-u)I, \\
b^T y' - \langle X', C \rangle &\ge & v+uM.
\end{array}
\]
\end{lemma}

\begin{proof}
For any strategy $\bar{Y} = \diag(X^{(2)},\diag(y^{(2)}),t^{(2)})$ 
of the second player, substituting $\bar{X} = \diag(X',y',t',u)$
into the payoff function~\eqref{eq:payoff_G_m} and rearranging the terms gives
\[
  \label{eq:payoff_Bob}
  \begin{array}{rcl}
  p(\bar{X},\bar{Y})& = & \langle t'C - \sum\limits_{i=1}^m y_i'A_i +uI, X^{(2)} \rangle 
  + \sum\limits_{i=1}^m y_i^{(2)}(\langle X', A_i \rangle - t'b_i +u)\\
  & & + \ t^{(2)}( b^Ty' - \langle X', C \rangle - uM) \, .
  \end{array}
\]
If the assertion were not true, then there would exist a strategy for player~2
which ensures a payoff to player 1 that is strictly smaller than the value 
$v$ of the game.
\end{proof}

A similar statement is true for optimal strategies of the second player.

\begin{lemma} \label{le:payoff_bound_2}
Given a pair $(\primal,\dual)$, 
assume that its modified Dantzig game $G_M$ has value $v\ge 0$ and an optimal 
strategy $\bar{Y} = \diag(X^*,y^*,t^*)$ for the second player.
Then
\[
  \label{eq:le:payoff_bound_2_v}
  \begin{array}{rcl}
    t^*b_i - \langle A_i,X^*\rangle  & \le & v \quad \text{for } i \in [m],\\
    \sum\limits_{i=1}^m y_i^*A_i -t^*C &\preceq & vI, \\
    \langle C, X^*\rangle - b^Ty^* &\le & v,\\
    \tr(X^*)+\sum\limits_{i=1}^m y_i^* -t^*M & \le & v.
  \end{array}
\]
\end{lemma}

Now we can prove Theorem \ref{th:reduction-unbounded}.

\begin{proof}[Proof of Theorem~\ref{th:reduction-unbounded}]
Consider a pair $(\primal,\dual)$ and its 
modified Dantzig game $G_M$. Assume that a strictly unbounded direction
exists in either $\primal$ or $\dual$. We already know that
the value $v$ of $G_M$ is positive.
Consider an optimal strategy of the form $\diag (X',y',t',u) $ for player 1 in $G_M$ and let $\diag(X^{(2)}, y^{(2)}, t^{(2)} )$ be some strategy for player 2.

We claim that $0 < u < 1$. To see this, consider the payoff of the first player
as described in \eqref{eq:payoff_G_m}. If $u=0$ then $\diag (X',y',t')$ would be a strategy of the unmodified Dantzig game $G$ which guarantees a positive payoff to player 1 even though the value of $G$ is zero. If $u=1$ then we obtain
$X'=0, y'=0, t'=0$, but this contradicts the optimality of the first player's
strategy in $G_M$. Namely, the second player can, for example, choose $t^{(2)}=1$ 
which leads to the negative payoff $-M$ for the first player.

Since $u>0$, for every optimal strategy $\diag(X^*,y^*,t^*)$
of the second player, the last diagonal strategy of the first player must be
a best response. Hence,
\begin{equation} \label{eq:v}
v \ = \ \tr(X^{*})+ \ones^T y^{*} - t^{*}M \ = \ (1-t^{*})- t^{*}M \ = \ 1-(1+M)t^{*},
\end{equation}
where the second equality holds because 
every strategy has trace~1.

Now we show that $u\ge v$. Assume that $v>u$. 
Then, again, $\diag (X',y',t')$ would be a strategy of the unmodified 
Dantzig game $G$ which guarantees a positive payoff to player 1 by Lemma~\ref{le:payoff_bound}.
This contradicts that $G$ has value 0.
Thus, we have $v\le u < 1$ and then
\eqref{eq:v} implies
$t^*>0$.

Next, we show that every optimal strategy for the first player is of the form $\diag(X',y',0,v)$, that is, 
$t'=0$ and $u=v$. Since we know that $u\ge v$, 
we need to prove $u \le v$.
Using Lemma \ref{le:payoff_bound}, we see
\begin{equation}\label{eq:u-v}
\begin{array}{rcl}
 t'C - \sum\limits_{i=1}^m y_i'A_i + (u-v)I &\succeq & 0, \\ [9pt] 
\langle X', A_i \rangle - t'b_i +(u-v) &\ge & 0 \quad \text{for } i \in [m],\\[9pt]
b^T y' - \langle X', C \rangle +(v-u)M &\ge & v(1+M)
\end{array}
\end{equation}
and in addition, the trace 1-property of strategies gives
$\tr(X') + \ones^T y' + t' + (u-v) = 1-v.$

$\primal_{\aux}$ attains its minimum $\bar{w}$ for some feasible point $(\bar{X},\bar{y},\bar{w})$, as we have already seen in Remark \ref{rem:strictly-dual}.
By Lemma \ref{le:payoff_bound_2}, $(X^*/t^*, y^*/t^*, v/t^*)$ is feasible for $\primal_{\aux}$ 
and therefore $\bar{w} \le v/t^* = \frac{v(M+1)}{1-v}$ due to \eqref{eq:v}.
Further,
\begin{eqnarray*}
v-u & \ge & \underbrace{(u-v)}_{\ge 0} 
  \underbrace{(\tr(\bar{X}) + \sum\nolimits_{i} \bar{y}_i -\bar{w} - M)}_{\le -1} 
   +\underbrace{t'}_{\ge 0} \underbrace{(\langle C,\bar{X}\rangle - b^T\bar{y} - \bar{w} )}_{\le 0} \\
& & + \ \langle \underbrace{X'}_{\succeq 0}, 
  \underbrace{\sum\nolimits_{i} \bar{y}_i A_i -\bar{w} I -C}_{\preceq 0}  \rangle
  + \sum\nolimits_i \underbrace{y_i'}_{\ge 0} 
    \underbrace{(b_i - \langle A_i, \bar{X} \rangle - \bar{w})}_{\le 0} \\
& = & \langle \bar{X}, (u-v)I + t'C - \sum\nolimits_i y_i'A_i\rangle + \sum\nolimits_i \bar{y}_i ((u-v)-t'b_i + \langle X',A_i \rangle)\\
& & -(u-v)M - \langle X', C\rangle + b^T y' 
  + \bar{w} (v-u-t'-\tr(X')- \ones^T y')\\
& \ge & 0 + v(M+1) + \bar{w}(v-1)\\
& \ge & 0 \, ,
\end{eqnarray*}
where the equality comes from rearranging
the terms and 
the penultimate inequality comes
from~\eqref{eq:u-v}.
This shows $u=v$ and, since in the preceding chain of inequalities all inequalities become
equations, further
\begin{equation} \label{eq:w^*}
\bar{w} \ = \ \frac{v(M+1)}{1-v}.
\end{equation}
Therefore, the inequalities \eqref{eq:u-v} give
\begin{equation} \label{eq:proof}
\begin{array}{rcl}
 t'C - \sum\limits_{i=1}^m y_i'A_i &\succeq & 0, \\ [9pt]
\langle X', A_i \rangle - t'b_i & \ge & 0 \quad \text{for } i \in [m], \\ [9pt]
b^T y' - \langle X', C \rangle &\ge & v(1+M)>0.
\end{array}
\end{equation}
Finally, we prove that $t'=0$. Assuming $t'>0$ gives
\begin{equation*}
p(\diag(X',y',0,0),\diag(X',y',0)) 
\ = \ \langle X' , \sum_{i=1}^m y_i'A_i \rangle - \langle X' , \sum_{i=1}^m y_i'A_i \rangle 
\ = \ 0.
\end{equation*}
However, then the inequalities \eqref{eq:proof} lead to a contradiction:
\begin{eqnarray*}
0 &= & \frac{1}{t'}\, p(\diag(X',y',0,0),\diag(X',y',0)) \\
& = & \sum_{i=1}^m y_i' \left\langle \frac{1}{t'} X', A_i \right\rangle - \left\langle \sum_{i=1}^m \frac{1}{t'} y_i'A_i , X' \right\rangle\\
&\ge & b^T y' - \left\langle C , X' \right\rangle \ > \ 0 \, ,
\end{eqnarray*}
where the positivity in the last step comes from the last line in~\eqref{eq:proof}.
Hence, $t'=0$. Then the inequalities~ \eqref{eq:proof} simplify to
\[
  \sum\nolimits_{i=1}^m y_i' A_i \preceq 0, \qquad
  \langle X', A_i \rangle  \geq 0 \ \text{ for } i \in [m], \qquad
  b^T y' - \left\langle C , X' \right\rangle > 0,
\]
with $X' \succeq 0$ and $y' \geq 0$. The strict inequality forces at least
one of the following to hold:
\begin{enumerate}[label=(\roman*)]
  \item $\left\langle C , X' \right\rangle < 0$, in which case $X'$ is an unbounded
  direction for $\primal$ and hence a certificate that $\dual$ is
  infeasible; or
  \item $b^T y' > 0$, in which case $y'$ is an unbounded direction
  for $\dual$ and hence a certificate that $\primal$ is infeasible.
\end{enumerate}
This concludes the proof.
\end{proof}

We capture the preconditions in the bounded and unbounded
case into a single constraint qualification.

\begin{definition}
\label{de:constraint-qualification}
Given a pair of SDPs $(\mathcal{P},\mathcal{D})$, we say that the
\emph{constraint qualification} holds if 
\begin{enumerate}
\item either $\mathcal{P}$ and $\mathcal{D}$ are bounded and a pair
of strongly optimal solutions of $\mathcal{P},\mathcal{D}$ exists
\item or there exists a strictly unbounded direction in 
  $\mathcal{P}$ or $\mathcal{D}$.
\end{enumerate}
\end{definition}

Using this definition, we can combine Theorem~\ref{th:reduction},
Theorem~\ref{th:reduction-unbounded} 
and~\cite[Theorem 4.1]{semi-games} into the following 
summarizing formulation.

\begin{cor}[Full equivalence under the constraint qualification]
\label{co:summary}

1. For each zero-sum semidefinite game $G$, the optimal strategies of
$G$ can be expressed as the solutions of a semidefinite program.

2. For each pair of SDPs $(\mathcal{P},\mathcal{D})$ which satisfies
the constraint qualification, any equilibrium of the modified Dantzig game
$G_M$ yields either an optimal point of the pair of SDPs or
a certificate that $\mathcal{P}$ or $\mathcal{D}$ is infeasible.
\end{cor}

\section{Bounds on the (primal/dual) solutions of an SDP\label{se:M}}
The equivalence theorems in Section~\ref{se:reduction}
(Theorems~\ref{th:reduction} and~\ref{th:reduction-unbounded})
require a solution bound $M$ satisfying condition~\eqref{eq:def:M}.
We show that, when the input data consists of rational numbers, then 
we can compute such a bound
explicitly by appealing to coordinate bounds for
the isolated solutions of the KKT polynomial system associated
with the SDP  \cite{emt-dmm-jsc-20}.

Let $A_1,\dots,A_m,C\in\mathbb Z^{n\times n}$ be symmetric matrices and
$b\in\mathbb Z^m$. 
Consider the following primal--dual semidefinite programming pair, 

\begin{equation}
\label{eq:SDP-pd}
\begin{array}{c}
\min_{X\in \sym_n}\ \langle C, X\rangle
\quad \text{s.t.}\quad
\langle A_i, X\rangle =b_i, i \in [m],  \text{ and } X\succeq 0, \\ [12pt]
\max_{y\in \R^m}\ b^\top y
\quad\text{s.t.}\quad
S:=C-\sum\nolimits_{i=1}^m y_iA_i\succeq 0.
\end{array}
\end{equation}

Note that
the normal form applied here slightly differs from
the one in \eqref{eq:SDP_primal} and \eqref{eq:SDP_dual}.
By ignoring the semidefinite inequalities, 
the corresponding KKT polynomial system, 
in the variables
\[
(X,S,y)\in \sym_n \times \sym_n \times \R^m
\]
is 
\begin{equation}
\label{eq:KKT}
(S_{\mathrm{KKT}}):\quad
\begin{cases}
\langle A_i,X\rangle-b_i=0, & i \in [m],\\[3pt]
S - C + \sum\nolimits_{i=1}^m y_i A_i = 0,\\[6pt]
XS=0.
\end{cases}
\end{equation}
Let
\[
\tau_0:=1+\max \Big \{ \lg |(A_i)_{i,j}|, \lg |C_{i,j}|, \lg |b_i|\Big\}
\]
be a uniform upper bound on the bitsize (all logarithms are base $2$) 
on the elements of all
the matrices and vectors appearing in $(S_{\mathrm{KKT}})$.
The characteristics of the system are as follows:
\begin{itemize}
\item the number of variables is
$ N=\dim(\sym_n)+\dim(\sym_n)+m=n(n+1)+m$,
\item it has more equations than unknowns (overdetermined); the number of equations is $p=m+\frac{n(n+1)}{2}+n^2=m+\frac{3n^2+n}{2}$, and
\item all the equations have maximal total degree at most $d=2$.
\end{itemize}

To apply the root bounds from \cite{emt-dmm-jsc-20}, we need a square system; as many equations as variables. 
For this we use a result by Giusti and Heintz \cite[Theorem~16]{emt-dmm-jsc-20},
that tells us that there exist $N$ generic integer linear combinations of the $p$ equations
forming a square polynomial system $(\widehat{S}_{\mathrm{KKT}})$.
Then, every irreducible component of the variety (zero set) 
of the new system, $\V(\widehat{S}_{\mathrm{KKT}})$, is either a component of the variety of the original system 
$\V(S_{\mathrm{KKT}})$ or an isolated point.
In particular, all isolated solutions of~\eqref{eq:KKT} are preserved.
If the coefficients of the original equations have bitsize at most $\tau$, then 
the squared-up system satisfies the explicit height bound 
\begin{equation}
\label{eq:tau_1}
\tau_1=\tau+N+\lg p
= \tau + (n(n+1)+m) + \lg(m+\tfrac{3n^2+n}{2}).
\end{equation}

Now we apply \cite[cor.~17]{emt-dmm-jsc-20} to obtain the following bounds for the isolated solutions of \eqref{eq:KKT}.

\begin{thm}[Explicit coordinate and separation bounds]
\label{thm:bounds}
Let $\gamma$ be an isolated solution of polynomial system
$(\widehat{S}_{\mathrm{KKT}})$ (and hence an isolated solution of
$({S}_{\mathrm{KKT}})$).
Let $N=n(n+1)+m$ and $\tau_1$ be given by~\eqref{eq:tau_1}.
Then, for every nonzero coordinate $\gamma_k\neq 0$,
it holds
\[
2^{-\eta_1}\le |\gamma_k|\le 2^{\eta_1},
\]
where
\[
\eta_1
=
\frac{N^2-N}{2}+2^N+N(\tau+N+2)2^{N-1}
= \sOO(n^4 \, m^2\, 2^{n(n+1)+m} \, \tau ),
\]
and the $\sOO(\cdot)$ notation ignores the (poly)logarithmic factors.
\end{thm}

\subsection{Reformulation of the auxiliary SDP}

To obtain the solution bound for the auxiliary SDP 
$\primal_{\aux}$, see \eqref{eq:w_defi_SDP_primal},
using the KKT system $(S_{\mathrm{KKT}})$,
we should rewrite it (and its dual) in the standard equality form 
that appears in \eqref{eq:SDP-pd}
and involves a single positive semidefinite matrix.

The SDP $\primal_{\aux}$ is transformed as follows.
We introduce slack variables
$s \in \R_+^m$, $Z \in \sym_n^+$, and $r \in \R_+$.
Then the constraints can be rewritten as
\begin{eqnarray*}
\langle A_i, X\rangle + w - s_i &=& b_i,\quad i\in[m], \\
\sum\nolimits_{i=1}^m y_i A_i - w I_n + Z &=& C, \\
\langle C, X\rangle - b^\top y - w + r &=& 0,
\end{eqnarray*}
where $X \succeq 0$, $Z \succeq 0$, $y \ge 0$, $s \ge 0$,
$w \ge 0$, and $r \ge 0$.
We now encode all variables into a single block-diagonal 
(positive semidefinite) matrix
\[
\mathcal{X}
:=
\diag (X, Z,  \diag(y), \diag(s), w, r ) \in\ \sym_{2n+2m+2}^+.
\]
In this setting, $\primal_{\aux}$, in standard primal SDP form,
becomes
\[
\min \ \langle \mathcal{C}, \mathcal{X} \rangle
\quad \text{s.t.}\quad
\langle \mathcal{A}_j, \mathcal{X} \rangle = \beta_j,\quad j \in [\bar m],
\quad
\mathcal{X} \succeq 0,
\]
where $\mathcal{C}$ is a matrix that  extracts the $w$-component from $\mathcal{X}$ 
and the matrices $\mathcal{A}_j$ encode the linear equalities.
Consequently, the number of equality constraints is
\[
\bar m = m + \frac{n(n+1)}{2} + 1,
\]
as it accounts for the $m$ scalar constraints, 
the $\frac{n(n+1)}{2}$ scalar equations arising from the matrix identity, and the final scalar constraint.
Finally, the dimension of the semidefinite variable is
$\mathcal{X} \in S_{2n+2m+2}^+$.
This reformulation allows us to apply the standard KKT conditions for semidefinite programming, leading to a polynomial system in the entries of $\mathcal{X}$ and of the dual slack matrix, to which 
we can apply the root bounds.

\begin{cor}[Coordinate bound for the auxiliary SDP]
\label{cor:coord-bd-Paux}
Let $\primal_{\aux}$ be as in~\eqref{eq:w_defi_SDP_primal}, with input data
$A_1, \dots, A_m, C \in \Z^{n\times n}$, and $b \in \Z^m$
having bitsize at most $\tau_0$. Consider its reformulation as a standard SDP over a single block variable
\[
\mathcal{X}\in S_{\bar n}^+, \quad \text{ where } \: \bar n:=2n+2m+2,
\]
with the number of equality constraints being 
\[
\bar m:=m+\frac{n(n+1)}{2}+1.
\]

Let $\gamma$ be an isolated solution of the squared-up KKT system associated with this formulation. Then, for every nonzero coordinate $\gamma_k$, it holds 
\[
2^{-\bar\eta_1}\le |\gamma_k|\le 2^{\bar\eta_1},
\]
where
\[
\bar N:=\bar n(\bar n+1)+\bar m,
\qquad
\bar p:=\bar m+\frac{\bar n(\bar n+1)}{2}+\bar n^2,
\qquad 
\bar\tau_1:=\tau_0+\bar N+\lg(\bar p),
\]
and
\[
\bar\eta_1
=
\frac{\bar N^2-\bar N}{2} + 2\bar N + \bar N(\bar\tau_1+\bar N+2)\,2^{\bar N-1}
= \tau_0 \, 2^{\sOO((n+m)^2)} ,
\]
where the $\sOO(\cdot)$ notation ignores the (poly)logarithmic factors.
In particular, every nonzero entry of the blocks
$ X, Z, y, s, w$, and $r$
of an isolated KKT solution is bounded in absolute value by $2^{\bar\eta_1}$.
\end{cor}

\begin{cor}[Solution bound for $\primal_{\aux}$]
\label{cor:M-bound-Paux}
Assume that the infimum of $\primal_{\aux}$ is attained
and that there exists an optimal solution
$ (X^*,Z^*,y^*,s^*,w^*,r^*)$
that corresponds to an isolated point (solution) of the KKT system.
Then, the entries of $X^*$ and $y^*$ satisfy
\[
\tr(X^*) \le n\,2^{\bar\eta_1},
\qquad
\mathbf{1}^\top y^* \le m\,2^{\bar\eta_1},
\]
and therefore
\[
\tr(X^*)+\mathbf{1}^\top y^*+1 \le (n+m)\,2^{\bar\eta_1}+1.
\]
In particular,
\[
M := (n+m)\,2^{\bar\eta_1}+1 = 2^{\tau_0 \, 2^{\sOO((n+m)^2)}}
\]
is a valid solution bound satisfying \eqref{eq:def:M}.
\end{cor}

\begin{example}
	\label{ex:opt-M-bd}
	The following example dating back to
    Khachiyan (see \cite{porkolab-khachiyan-1997} 
    or \cite{pataki-touzov-2024})
    demonstrates that asymptotically exponential behavior of the bitsize of the solutions of SDP is nearly optimal.
Set $B:=2^{\tau}$ and consider the SDP
\[
\begin{aligned}
	\min \quad & x_n \\
	\text{s.t. } \quad &
	\begin{pmatrix}
		x_i & x_{i-1}\\
		x_{i-1} & B
	\end{pmatrix} \succeq 0,
	\quad i \in [n], \\
	& x_0=\tfrac{1}{2} .
\end{aligned}
\]
The input (bit)size is polynomial in $n + \tau$.
Each constraint, there are $n$, is equivalent to $x_i \ge \frac{x_{i-1}^2}{B}$.
At the optimum value of the SDP, 
all the inequalities become equalities, 
because the objective function minimizes $x_n$.
Thus, the optimal values are $x_i^*=\tfrac{(x_{i-1}^*)^2}{B}$.
Using induction, we deduce that the optimal value is 
\[
x_n^*=\frac{x_0^{2^n}}{B^{2^n-1}}
=\frac{(1/2)^{2^n}}{(2^{\tau})^{2^n-1}}
=2^{-(\tau+1)2^n+ \tau}.
\]
\end{example}

\section{Examples\label{se:examples}}

By Theorem~\ref{th:reduction},
our reduction produces strongly optimal solutions of a pair $(\primal,\dual)$ whenever such solutions exist. 
By Theorem~\ref{th:reduction-unbounded},
we recover an unbounded direction of $\primal$ or $\dual$ whenever a strictly unbounded direction exists. We have argued that those are the most important cases to consider, and both these cases are fully covered by our reduction. There are, however, some pathological pairs of feasible SDPs $(\primal,\dual)$ where a duality gap exists
or the optimal values may not be attained and some pathological
pairs where not both SDPs are feasible but a strictly unbounded
direction does not exist. We provide some of examples of those 
kinds of pairs $(\primal,\dual)$ in this section.

Our first example considers a pair of SDPs $(\primal,\dual)$ where $\primal$ and $\dual$ are both infeasible and our reduction yields a dual unbounded direction, which is a certificate that $\primal$ is infeasible.

\begin{example} \label{ex:SDP_both_infeasible_works}
Consider the pair of SDPs given by
\begin{equation}
\begin{array}{rcl}
\primal & : & \inf \left\{ \left\langle \begin{pmatrix}
0 & -1\\
-1 & 0
\end{pmatrix}, X\right\rangle \, : \, \left\langle \begin{pmatrix}
-1 & 0\\
0 & 0
\end{pmatrix}, X\right\rangle \ge 1, \, X\succeq 0 \right\}, \\
\dual & : & \sup \left\{ y \, : \, y \begin{pmatrix}
-1 & 0\\
0 & 0
\end{pmatrix} \preceq \begin{pmatrix}
0 & -1\\
-1 & 0
\end{pmatrix}, \, y\ge 0 \right\} .
\end{array}
\end{equation}
It is not difficult to see that neither $\primal$ nor $\dual$ is feasible.
The infimum of the auxiliary program
\begin{equation*} 
\begin{array}{rl}
\primal_{aux} \ : \ \inf & \left\{ w \ : \
\left\langle \begin{pmatrix}
-1 & 0\\
0 & 0
\end{pmatrix}, X\right\rangle + w \ge 1, \;
\begin{pmatrix}
y & -1\\
-1 & 0
\end{pmatrix} + w I_2 \succeq 0, \right. \\
& \quad \left. \left\langle \begin{pmatrix}
0 & -1\\
-1 & 0
\end{pmatrix},X \right\rangle - y -w \le 0, \, 
X \succeq 0, \, y \ge0, \, w\ge 0 \right\}
\end{array}
\end{equation*}
is attained for $(X^*,y^*,w^*)=(0,0,1)$, so we can chose $M=1$.
The modified Dantzig game $G_M$ is now defined by the payoff
$p(\bar{X},\bar{Y})$ being equal to
\begin{eqnarray*}
& &  y^{(1)} \left(t^{(2)} -\left\langle \begin{pmatrix}
-1 & 0\\
0 & 0
\end{pmatrix} , X^{(2)} \right\rangle \right) + 
\left\langle X^{(1)}, y^{(2)}\begin{pmatrix}
-1 & 0\\
0 & 0
\end{pmatrix} -  t^{(2)}\begin{pmatrix}
0 & -1\\
-1 & 0
\end{pmatrix} \right\rangle \\
& & + t^{(1)}\left(\left\langle \begin{pmatrix}
0 & -1\\
-1 & 0
\end{pmatrix}, X^{(2)}\right\rangle - y^{(2)})
 + u (\tr(X^{(2)})+ y^{(2)} - t^{(2)})\right).
\end{eqnarray*}
An optimal strategy for the second player is $\diag(X^{(2)},y^{(2)},t^{(2)})= \diag(0,\frac{2}{3}, \frac{1}{3})$ and an optimal strategy for the first player is $\diag(X^{(1)},y^{(1)},t^{(1)},u)= \diag(0,\frac{2}{3},0, \frac{1}{3})$. Note that $u=\frac{1}{3}$ is indeed the value of the game and further, $y^{(1)}=\frac{2}{3}$ is an unbounded direction of $\dual$.
\end{example}

This shows that our reduction can also hold in some cases that are not covered by Theorem \ref{th:reduction} and Theorem \ref{th:reduction-unbounded}. In particular, our reduction gives the desired result for a pair $(\primal,\dual)$ whenever the infimum in $\primal_{aux}$ is attained. This follows directly from the proof of Theorem \ref{th:reduction-unbounded}. However, there are also pairs
$(\primal,\dual)$ where the infimum in $\primal_{aux}$ is not attained and we still recover an unbounded direction.
To see this, consider a slight variation on Example \ref{ex:SDP_both_infeasible_works}, where the infimum in $\primal_{aux}$ is not attained.
\begin{example} 
Consider the pair
\begin{equation}
\begin{array}{rcl}
\primal & : & \inf \left\{ \left\langle \begin{pmatrix}
-2 & -1\\
-1 & -2
\end{pmatrix}, X\right\rangle \, : \, \left\langle \begin{pmatrix}
-1 & 0\\
0 & 0
\end{pmatrix}, X\right\rangle \ge 1, \, X\succeq 0 \right\}, \\
\dual & : & \sup \left\{ y \, : \, \begin{pmatrix}
-2+y & -1\\
-1 & -2
\end{pmatrix}\succeq 0, \, y\ge 0 \right\} .
\end{array}
\end{equation}
It is not difficult to see that neither the primal nor the dual program is feasible.
The infimum
\begin{equation*} 
\begin{array}{rl}
\primal_{aux} \ : \ \inf & \left\{ w \ : \
\left\langle \begin{pmatrix}
-1 & 0\\
0 & 0
\end{pmatrix}, X\right\rangle + w \ge 1, \;
\begin{pmatrix}
-2 + y & -1\\
-1 & -2
\end{pmatrix} + w I_2 \succeq 0, \right. \\
& \quad \left. \left\langle \begin{pmatrix}
-2 & -1\\
-1 & -2
\end{pmatrix},X \right\rangle - y -w \le 0, \, 
X \succeq 0, \, y \ge0, \, w\ge 0 \right\}
\end{array}
\end{equation*}
is equal to two but is not attained. In fact, for every $\varepsilon>0$, the point $(X^*,y^*,w^*)=(0,y^*,2+\varepsilon)$ is feasible for sufficiently large $y^*$, but the second condition can never be met for $w^*=2$. Choosing $M=1$, we can define the modified Dantzig game $G_M$ similarly as before and notice the same optimal strategies for both players as before.
In particular, we recover an unbounded direction of $\dual$ even though the infimum in $\primal_{aux}$ was not attained.
\end{example}

Now, let us consider an example of a degenerate pair of SDPs $(\primal,\dual)$ where both optimal values are attained but there is a duality gap. The previous reduction from \cite{semi-games} fails to give us any information about the optimal solutions at all.
In contrast, for any suitable choice of a solution bound $M$, the value of the game $G_M$ is strictly positive which proves that no pair of strongly optimal solutions exist for $(\primal,\dual)$.  This particular pair was introduced in \cite{dimou-equiv-2024}.

\begin{example} 
Consider the pair of SDPs
\begin{equation}
\begin{array}{rl}
\primal \ : \ \inf & \left\{ \left\langle \begin{pmatrix}
0 & 0 & 0\\
0 & 1& 0\\
0 & 0 & 1
\end{pmatrix}, X\right\rangle \, : \, \left\langle \begin{pmatrix}
0& -1 & 0\\
-1 & 0 & 0\\
0 & 0 & 1
\end{pmatrix}, X\right\rangle \ge 1, \right.\\
 &\left. \left\langle \begin{pmatrix}
0& 0 & 0\\
0 & -1 & 0\\
0 & 0 & 0
\end{pmatrix}, X\right\rangle \ge 0, \, X\succeq 0 \right\}\\
\dual \ : \ \sup & \left\{ y_1 \, : \, \begin{pmatrix}
0 & y_1 & 0\\
y_1 & 1 +y_2 & 0\\
0 & 0 & 1 -y_1
\end{pmatrix}\succeq 0, \, y_1,y_2\ge 0 \right\} .
\end{array}
\end{equation}
The infimum is one and it is attained by choosing, say, 
$X=\left( \begin{smallmatrix}
0 & 0 & 0\\
0 & 0 & 0\\
0 & 0 & 1
\end{smallmatrix} \right)$.
The supremum is zero and is attained for any $y\ge 0$ with $y_1=0$.
The infimum of the auxiliary program 
\begin{equation*}
\begin{array}{rl}
\primal_{aux} \ : \ \inf & \Big \{ w \, : \, \left\langle \begin{pmatrix}
0& -1 & 0\\
-1 & 0 & 0\\
0 & 0 & 1
\end{pmatrix}, X\right\rangle +w \ge 1, \left\langle \begin{pmatrix}
0& 0 & 0\\
0 & -1 & 0\\
0 & 0 & 0
\end{pmatrix}, X\right\rangle +w \ge 0, \\
& \begin{pmatrix}
0 & y_1 & 0\\
y_1 & 1+y_2 & 0\\
0 & 0 & 1-y_1
\end{pmatrix} + wI_3 \succeq 0,\, \left\langle \begin{pmatrix}
0 & 0 & 0\\
0 & 1 & 0\\
0 & 0 & 1
\end{pmatrix}, X\right\rangle - y_1 -w \le 0,\\
& w\ge0, \ X\succeq 0, \ y_1,y_2\ge 0 \Big\}.
\end{array}
\end{equation*}
is zero but is not attained. In fact, we can see that any value 
$\varepsilon>0$ can be attained by choosing $w=\varepsilon,\; y=0$ and $X=\left( \begin{smallmatrix}
\frac{1}{\varepsilon} & -1 & 0\\
-1 & \varepsilon & 0\\
0 & 0 & 0
\end{smallmatrix} \right)$.
Our construction applied to the given SDP and $M=1$ yields the payoff
\begin{equation*}
\begin{array}{rcl}
p(\bar{X},\bar{Y})&= & y_1^{(1)} (t^{(2)} + 2 x_{12}^{(2)} - x_{33}^{(2)} ) + y_2^{(1)} x_{22}^{(2)} \\
& & + \left\langle X^{(1)}, \begin{pmatrix}
0 & -y_1^{(2)} & 0\\
-y_1^{(2)} & -y_2^{(2)} - t^{(2)} & 0\\
0 & 0 & y^{(2)}_1-t^{(2)}
\end{pmatrix}
\right\rangle \\
& & + t^{(1)}(x_{22}^{(2)} + x_{33}^{(2)} - y_1^{(2)})
 + u ( x_{11}^{(2)}+ x_{22}^{(2)}+ x_{33}^{(2)} + y_1^{(2)} + y_2^{(2)} - t^{(2)}) \, .
\end{array}
\end{equation*}
We see that for the game $G$ without the last additional term, the strategies 
\[
y^{(2)} = \left(0,\frac{1}{2}\right), 
\quad 
X^{(2)} = \frac{1}{2}
\begin{pmatrix}
1 & 0 & 0\\[3pt]
0 & 0 & 0\\[3pt]
0 & 0 & 0
\end{pmatrix},\quad t^{(2)}=0
\]
are indeed optimal and the game value is zero.

In our modified game, this is no longer the case as the payoff to the first player is now equal to $1$ for this strategy if player 1 plays $u=1$. Instead, the game value is positive and we see that no pair of strongly optimal solutions exists.
However, since there is a positive duality gap, we do not actually recover the optimal solutions from our reduction and we do not recover unbounded directions since those do not exist in this case.
\end{example}

\section{Outlook\label{se:outlook}}

We have elevated the equivalence of linear programming
and zero-sum matrix games to the case of semidefinite programming
and zero-sum semidefinite games. Dimou \cite{dimou-equiv-2024}
extended the underlying principles from \cite{semi-games}
to the case of conic games. At least certain ingredients 
of the present work also extend to the case of
conic games. However, we leave it as a future task to 
consider the generalization of our results to conic games or to other 
generalized classes in detail.

\medskip

\noindent{\bf Acknowledgments.}
The research was partially supported by the joint PROCOPE project ``Quantum 
games and polynomial optimization'' of the MEAE/MESR and the DAAD (57753345).

\bibliographystyle{abbrv}

\bibliography{Equibib}

\end{document}